\documentclass{amsart}
\usepackage{amssymb}

\theoremstyle{plain}
\newtheorem{theorem}{Theorem}
\newtheorem{corollary}[theorem]{Corollary}

\newtheorem{lemma}[theorem]{Lemma}

\numberwithin{equation}{section}

\begin{document}

\title[Identities for Bernoulli polynomials]{Identities for Bernoulli polynomials related to multiple Tornheim zeta functions}

\author{Karl Dilcher}
\address{Department of Mathematics and Statistics\\
         Dalhousie University\\
         Halifax, Nova Scotia, B3H 3J5, Canada}
\email{dilcher@mathstat.dal.ca}

\author{Armin Straub}
\address{Department of Mathematics and Statistics\\
University of South Alabama\\
Mobile, AL 36688, USA}
\email{straub@southalabama.edu}

\author{Christophe Vignat}
\address{LSS-Supelec, Universit\'e Paris-Sud, Orsay, France and Department of
Mathematics, Tulane University, New Orleans, LA 70118, USA}
\email{cvignat@tulane.edu}

\thanks{Research supported in part by the Natural Sciences and Engineering
        Research Council of Canada, Grant \# 145628481}

\date{}

\setcounter{equation}{0}

\begin{abstract}
We show that each member of a doubly infinite sequence of highly nonlinear 
expressions of Bernoulli polynomials, which can be seen as linear combinations 
of certain higher-order convolutions, is a multiple of a specific product of
linear factors. The special case of Bernoulli numbers has important 
applications in the study of multiple Tornheim zeta functions. The proof of
the main result relies on properties of Eulerian polynomials and higher-order
Bernoulli polynomials.
\end{abstract}

\maketitle

\section{Introduction}

Various convolution identities for Bernoulli polynomials and related numbers
and polynomials have attracted considerable attention in recent years. In
connection with a detailed study of multiple Tornheim zeta functions,
the first author \cite{Di,DT} recently obtained what appears
to be a new type of identity, namely
\begin{equation}\label{1.1}
\sum_{m=1}^n\binom{n+1}{m}\sum_{\substack{j_1,\ldots,j_m\geq 1\\j_1+\dots+j_m=n}}
\prod_{i=1}^m\frac{B_{j_i}(z)}{j_i!} 
= \frac{1}{n!}\prod_{j=1}^n\bigl((n+1)z-j\bigr),
\end{equation}
for $n\geq 1$. Here $B_k(z)$ is the $k$th {\it Bernoulli polynomial\/}, which 
can be defined by the generating function 
\begin{equation}\label{1.2}
\frac{xe^{zx}}{e^x-1} = \sum_{k=0}^\infty B_k(z)\frac{x^k}{k!},\qquad |x|<2\pi,
\end{equation}
and the $k$th {\it Bernoulli number\/} is defined by $B_k:=B_k(0)$, $k\geq 0$.
The first few Bernoulli polynomials are listed in Table~2 in Section~2 below.

A notable feature of the identity \eqref{1.1} is the fact that an easy linear
combination of convolutions results in a product of $n$ monomials. When the
Bernoulli polynomials on the left are replaced by the corresponding Bernoulli
{\it numbers\/}, i.e., setting $z=0$, then the right-hand side is simply 
$(-1)^n$.

Also useful in connection with multiple Tornheim zeta functions is an identity
that involves a generalization of the left-hand side of \eqref{1.1}. For all
integers $n\geq 1$ and $k\geq 0$, we define
\begin{equation}\label{1.3}
S_{n,k}(z):=
\sum_{m=1}^n\binom{n+1}{m}k!^{n-m}(-1)^{km}
\sum_{\substack{j_1,\ldots,j_m\geq 0\\j_1+\dots+j_m=(k+1)(n-m)+k}}
\prod_{i=1}^m\frac{B_{k+1+j_i}(z)}{j_i!(k+1+j_i)}.
\end{equation}
When $k=0$, a simple shift in indexing shows that $S_{n,0}(z)$ is the left-hand
side of \eqref{1.1}. The methods used in \cite{Di} for $k=0$ do not appear to 
apply when $k\geq 1$. It is the purpose of this paper to deal with this case.
To motivate our main result, we list the first few cases of $S_{n,1}(z)$
and $S_{n,2}(z)$ in Table~1.

\bigskip
\begin{center}
{\renewcommand{\arraystretch}{1.2}
\begin{tabular}{|r||l|l|}
\hline
$k$ & $n$ & $S_{n,k}(z)$ \\
\hline
1 & 1 & $\frac{-1}{3}z(z-1)(2z-1)$ \\
1 & 2 & $\frac{1}{20}z(z-1)(3z-1)(3z-2)(2z-1)$ \\
1 & 3 & $\frac{-1}{105}z(z-1)(4z-1)(2z-1)(4z-3)(z^2-z+1)$ \\
1 & 4 & $\frac{-1}{18144}z(z-1)(5z-1)(5z-2)(5z-3)(5z-4)(2z-1)(13z^2-13z-6)$ \\
\hline
2 & 1 & $\frac{1}{30}z(z-1)(2z-1)(3z^2-3z-1)$ \\
2 & 2 & $\frac{1}{160}z^2(z-1)^2(3z-1)(3z-2)(7z^2-7z-2)$ \\
2 & 3 & $\frac{1}{20790}z(z-1)(4z-1)(2z-1)(4z-3)(321z^6-\cdots-3)$ \\
2 & 4 & $\frac{1}{16765056}z^2(z-1)^2(5z-1)\cdots(5z-4)(19302z^6-\cdots-348)$ \\
\hline
\end{tabular}}

\medskip
{\bf Table~1}: $S_{n,1}(z)$ and $S_{n,2}(z)$ for $1\leq n\leq 4$.
\end{center}
\bigskip

We see that for both $k=1$ and $k=2$, the polynomial $S_{n,k}(z)$ is divisible 
by $z((n+1)z-1)\cdots((n+1)z-(n+1))$. It is the main purpose of this paper to
prove that this observation is in fact true for all $n$ and $k$.

\begin{theorem}\label{thm:1}
For all integers $n\geq 1$ and $k\geq 1$, the polynomial $S_{n,k}(z)$ 
satisfies 
\begin{equation}\label{1.4}
S_{n,k}(1-z) = (-1)^{(k+1)(n+1)-1}S_{n,k}(z),
\end{equation}
and is divisible by 
\[
z\prod_{j=1}^{n+1}\bigl((n+1)z-j\bigr).
\]
\end{theorem}

As an immediate consequence of Theorem~\ref{thm:1} we obtain a corresponding 
statement about Bernoulli {\it numbers\/}, which we can phrase as follows.

\begin{corollary}
For all integers $n\geq 1$ and $k\geq 1$ we have $S_{n,k}(0)=S_{n,k}(1)=0$,
and when at least one of $k$ and $n$ is odd, then $S_{n,k}(\frac{1}{2})=0$.
\end{corollary}

In particular, since $B_m(0)=B_m$, this means that the right-hand side of 
\eqref{1.3}, with Bernoulli polynomials replaced by Bernoulli numbers, is 0
for all $n\geq 1$ and $k\geq 1$.

In order to prove Theorem~\ref{thm:1}, we define an auxiliary power series
in Section~2 and prove some lemmas involving this series. In Section~3 we 
complete the proof of Theorem~\ref{thm:1}. Section~4 contains a few further
results, and we conclude this paper with some additional remarks in Section~5.

\section{Some lemmas}

We introduce a specific power series as an auxiliary function. For any integer
$k\geq 0$ and complex variable $z$ we define
\begin{equation}\label{2.1}
F_k(x,z) := 1+(-1)^k\frac{x^{k+1}}{k!}\sum_{m=0}^\infty 
\frac{B_{m+k+1}(z)}{m+k+1}\frac{x^m}{m!}.
\end{equation}
This series has the same radius of convergence, $2\pi$, as \eqref{1.2}.
Using the reflection identity for Bernoulli polynomials, namely 
$B_n(1-x)=(-1)^nB_n(x)$ (see, e.g., \cite[24.4.3]{DLMF}), it is easy to verify
that
\begin{equation}\label{2.1a}
F_k(-x,1-z) = F_k(x,z).
\end{equation}
The function $F_k(x,0)$, i.e., the case where on the right of \eqref{2.1} we
have Bernoulli {\it numbers} instead of polynomials, has previously been 
studied and applied by several authors; see \cite{CW}, \cite{Ei}, and 
\cite{On}. 

In what follows we denote the coefficient of $x^k$, $k\geq 0$, in a power 
series $f(x)$ by $[x^k]f(x)$. We can now state and prove the following result.

\begin{lemma}\label{lem:3}
For all $n\geq 1$ and $k\geq 0$ we have
\begin{equation}\label{2.2}
\frac{S_{n-1,k}(z)}{k!^{n-1}} = [x^{(k+1)n-1}]F_k(x,z)^n. 
\end{equation}
\end{lemma}

\begin{proof}
We rewrite \eqref{1.3} as
\begin{equation}\label{2.3}
\frac{S_{n,k}(z)}{k!^n} = \sum_{m=1}^n\binom{n+1}{m}S_{n,k,m}(z)
= \sum_{m=0}^{n+1}\binom{n+1}{m}S_{n,k,m}(z),
\end{equation}
where
\[
S_{n,k,m}(z):=
\sum_{\substack{j_1,\ldots,j_m\geq 1\\j_1+\dots+j_m=k(n+1-m)+n}}
\prod_{i=1}^m\frac{(-1)^kB_{j_i+k}(z)}{k!(j_i-1)!(j_i+k)}.
\]
We now define
\begin{equation}\label{2.4}
G_k(x,z) := \frac{(-1)^k}{k!}\sum_{j=1}^\infty
\frac{B_{j+k}(z)}{(j-1)!(j+k)}x^j
\end{equation}
and first observe that
\begin{equation}\label{2.5}
S_{n,k,m}(z)=[x^{k(n+1-m)+n}]G_k(x,z)^m.
\end{equation}
Then Cauchy's integral formula, applied to the right of \eqref{2.5}, gives
\[
S_{n,k,m}(z)=\frac{1}{2\pi i}\int_{\gamma}\frac{G_k(x,z)^m}{x^{k(n+1-m)+n+1}}dx,
\]
where the contour $\gamma$ traverses, for instance, a circle around the origin
with small enough radius, once in the positive direction. Then, after replacing
$n$ by $n-1$, we get with \eqref{2.3}, 
\begin{align*}
\frac{S_{n-1,k}(z)}{k!^{n-1}} &= \frac{1}{2\pi i}\sum_{m=0}^n\binom{n}{m}
\int_{\gamma}\frac{G_k(x,z)^m}{x^{k(n-m)+n}}dx \\
&= \frac{1}{2\pi i}\int_{\gamma}\frac{1}{x^{(k+1)n}}
\sum_{m=0}^n\binom{n}{m}x^{km}G_k(x,z)^mdx \\
&= \frac{1}{2\pi i}\int_{\gamma}\frac{(1+x^kG_k(x,z))^n}{x^{(k+1)n}}dx \\
&= [x^{(k+1)n-1}](1+x^kG_k(x,z))^n,
\end{align*}
where we have used Cauchy's integral formula again. Finally, since by 
\eqref{2.1} and \eqref{2.4} we have
\[
1+x^kG_k(x,z) = F_k(x,z),
\]
this proves our lemma.
\end{proof}

In what follows, we require the {\it Eulerian polynomials} which can be defined
by the generating function
\begin{equation}\label{2.6}
\frac{1-y}{1-ye^{(1-y)t}} = \sum_{k=0}^\infty A_k(y)\frac{t^k}{k!};
\end{equation}
see, e.g., \cite[p.~244]{Co}.  The first few Eulerian polynomials are 
listed in Table~2.

\bigskip
\begin{center}
{\renewcommand{\arraystretch}{1.2}
\begin{tabular}{|r||l|l|}
\hline
$k$ & $B_k(z)$ & $A_k(y)$ \\
\hline
0 & 1 & 1 \\
1 & $z-\tfrac{1}{2}$ & $y$ \\
2 & $z^2-z+\tfrac{1}{6}$ & $y^2+y$ \\
3 & $z^3-\tfrac{3}{2}z^2+\tfrac{1}{2}z$ & $y^3+4y^2+y$ \\
4 & $z^4-2z^3+z^2-\tfrac{1}{30}$ & $y^4+11y^3+11y^2+y$ \\
5 & $z^5-\tfrac{5}{2}z^4+\tfrac{5}{3}z^3-\tfrac{1}{6}z$ 
& $y^5+26y^4+66y^3+26y^2+y$ \\
6 & $z^6-3z^5+\tfrac{5}{2}z^4-\tfrac{1}{2}z^2+\tfrac{1}{42}$ 
& $y^6+57y^5+302y^4+302y^3+57y^2+y$ \\
\hline
\end{tabular}}

\medskip
{\bf Table~2}: $B_k(z)$ and $A_k(y)$ for $0\leq k\leq 6$.
\end{center}
\bigskip

The Eulerian polynomials are self-reciprocal (or palindromic), and we can write
\begin{equation}\label{2.7}
A_k(y) = \sum_{j=0}^kA(k,j)y^j\quad (k\geq 0).
\end{equation}
The coefficients $A(k,j)$ are the well-known {\it Eulerian numbers\/}, which
have important combinatorial interpretations. We also require the generating
function 
\begin{equation}\label{2.8}
\frac{A_k(y)}{(1-y)^{k+1}} = \sum_{m=0}^\infty m^ky^m;
\end{equation}
see, e.g., \cite[p.~245]{Co}. We can now rewrite the sequence of functions 
$F_k(x,z)$ that was defined in \eqref{2.1}.

\begin{lemma}\label{lem:4}
For all $k\geq 1$ we have
\begin{equation}\label{2.9}
F_k(x,z)=\bigg(\frac{x}{e^x-1}\bigg)^{k+1}e^{xz}
\sum_{j=0}^k(1-e^x)^j\frac{A_{k-j}(e^x)}{(k-j)!}\frac{z^j}{j!}.
\end{equation}
\end{lemma}

\begin{proof}
Applying a variant of a method used in \cite{Ei}, we find with \eqref{1.2},
\begin{align}
\sum_{m=1}^\infty e^{(z-m)x} &= e^{zx}\frac{e^{-x}}{1-e^{-x}}
= \frac{e^{zx}}{e^x-1} \label{2.10} \\
&= \frac{1}{x}+\frac{1}{x}\bigg(\sum_{m=0}^\infty B_m(z)\frac{x^m}{m!}-1\bigg)\nonumber\\
&= \frac{1}{x}+\sum_{m=0}^\infty\frac{B_{m+1}(z)}{m+1}\cdot\frac{x^m}{m!}.\nonumber
\end{align}
Note that the expressions in \eqref{2.10} have simple poles at $x=0$. Taking
the $k$th derivative with respect to $x$ of both sides of \eqref{2.10}, we get
\[
\sum_{m=1}^\infty(z-m)^k e^{(z-m)x} = (-1)^k\frac{k!}{x^{k+1}}
+\sum_{m=0}^\infty\frac{B_{m+k+1}(z)}{m+k+1}\cdot\frac{x^m}{m!}.
\]
Comparing this with \eqref{2.1}, we immediately see that
\begin{equation}\label{2.11}
F_k(x,z) = \frac{x^{k+1}}{k!}\sum_{m=1}^\infty(m-z)^k e^{(z-m)x}.
\end{equation}
Expanding $(m-z)^k$ as a binomial sum and changing the order of summation, we
get
\begin{equation}\label{2.12}
F_k(x,z) = x^{k+1}e^{zx}\sum_{j=0}^k\frac{(-z)^j}{j!(k-j)!}
\sum_{m=1}^\infty m^{k-j}(e^{-x})^m.
\end{equation}
Now by \eqref{2.8} and the palindromic property of the polynomials $A_k(y)$ we
have for $0\leq j\leq k-1$,
\begin{align*}
\sum_{m=1}^\infty m^{k-j}(e^{-x})^m &= \frac{A_{k-j}(e^{-x})}{(1-e^{-x})^{k-j+1}}
= \frac{e^{-x(k-j+1)}A_{k-j}(e^x)}{(1-e^{-x})^{k-j+1}} \\
& = \frac{(e^x-1)^j}{(e^x-1)^{k+1}}A_{k-j}(e^x),
\end{align*}
while for $j=k$ we use the first line of \eqref{2.10} with $z=0$, obtaining
\[
\sum_{m=1}^\infty(e^{-x})^m = \frac{1}{e^x-1}
= \frac{(e^x-1)^k}{(e^x-1)^{k+1}}A_0(e^x).
\]
If we combine these last two identities with \eqref{2.12}, we immediately get
\eqref{2.9}.
\end{proof}

\noindent
{\bf Remarks.} (1) An important special function, the {\it polylogarithm\/},
is defined by
\[
Li_s(z) = \sum_{n=1}^\infty\frac{z^n}{n^s},
\]
which for a fixed $s\in{\mathbb C}$ defines an analytic function of $z$ for
$|z|<1$; see, e.g., \cite[25.12.10]{DLMF}. Comparing this with \eqref{2.8}, 
we get
\[
A_k(y) = (1-y)^{k+1}Li_{-k}(y),
\]
and if we set $y=e^x$ and replace $k$ by $k-j$, we see that \eqref{2.9}
simplifies to
\begin{equation}\label{2.13}
F_k(x,z)=(-x)^{k+1}e^{xz}\sum_{j=0}^k\frac{Li_{j-k}(e^x)}{(k-j)!}\frac{z^j}{j!}.
\end{equation}
Note that the terms $e^x-1$ in \eqref{2.9} have disappeared.

(2) For an alternative approach to Lemma~\ref{lem:4}, see Part~4 of Section~5.

\section{Proof of Theorem~\ref{thm:1}}

By Lemma~\ref{lem:3} we need to determine the coefficient of $x^{(k+1)n-1}$ in
$F_k(x,z)^n$. To do so, we first change the order of summation in \eqref{2.9}
and obtain
\begin{equation}\label{3.1}
F_k(x,z)=\bigg(\frac{x}{e^x-1}\bigg)^{k+1}e^{xz}(z(1-e^x))^k
\sum_{j=0}^k\frac{A_j(e^x)}{(k-j)!}\cdot\frac{(z(1-e^x))^{-j}}{j!}.
\end{equation}
Taking the $n$th power of the sum in \eqref{3.1} gives
\[
\sum_{\nu\geq 0}(z(1-e^x))^{-\nu}
\sum_{\substack{j_1,\ldots,j_n\geq 0\\j_1+\dots+j_n=\nu}}
\frac{A_{j_1}(e^x)}{(k-j_1)!j_1!}\cdots\frac{A_{j_n}(e^x)}{(k-j_n)!j_n!},
\]
and the $n$th power of \eqref{3.1} then becomes
\begin{align}
F_k(x,z)^n &= \bigg(\frac{x}{e^x-1}\bigg)^{n(k+1)}\frac{e^{nxz}}{(k!)^n}
\sum_{\nu\geq 0}(z(1-e^x))^{nk-\nu} \label{3.2}\\
&\quad\times\sum_{\substack{j_1,\ldots,j_n\geq 0\\j_1+\dots+j_n=\nu}}
\binom{k}{j_1}A_{j_1}(e^x)\cdots\binom{k}{j_n}A_{j_n}(e^x).\nonumber
\end{align}
Now let $S_{k,\nu}^{(n)}(e^x)$ be the multiple sum in \eqref{3.2}. It is not
difficult to compute the first few of these sums:
\begin{align*}
S_{k,0}^{(n)}(e^x) &= 1,\\
S_{k,1}^{(n)}(e^x) &= \frac{nk}{1!}e^x,\\
S_{k,2}^{(n)}(e^x) &= \frac{nk}{2!}e^x\left[(k-1)+(kn-1)e^x\right],\\
S_{k,3}^{(n)}(e^x) &= \frac{nk}{3!}e^x\left[(k-1)(k-2)+(k-1)(3kn+k-8)e^x
+(kn-1)(kn-2)e^{2x}\right].
\end{align*}
These polynomials will be further investigated in Section~4.

\begin{proof}[Proof of Theorem~\ref{thm:1}]
The identity \eqref{1.4} follows from Lemma~\ref{lem:3} and \eqref{2.1a}, with
$n$ replaced by $n+1$.

For the remainder of the proof we rewrite \eqref{3.2} as
\begin{equation}\label{3.4}
F_k(x,z)^n = \bigg(\frac{x}{e^x-1}\bigg)^{n(k+1)}\frac{e^{nxz}}{(k!)^n}
\sum_{\nu=0}^{nk}(z(1-e^x))^{nk-\nu}S_{k,\nu}^{(n)}(e^x).
\end{equation}
We first consider the sum over $\nu$. For $\nu=0$, the corresponding summand
is simply
\begin{equation}\label{3.5}
z^{nk}\left(1-nke^x+\dots+(-1)^{nk}e^{nkx}\right),
\end{equation}
while for $1\leq \nu\leq nk$, the summand is
\begin{equation}\label{3.6}
z^{n(k+1)}(1-e^x)^{nk-\nu}S_{k,\nu}^{(n)}(e^x)
= z^{n(k+1)}\sum_{j=1}^{nk}d_j^{(\nu)}e^{jx},
\end{equation}
for certain coefficients $d_j^{(\nu)}$. Altogether, this and \eqref{3.5},
combined with \eqref{3.4}, gives
\begin{align}
F_k(x,z)^n &= \bigg(\frac{x}{e^x-1}\bigg)^{n(k+1)}\frac{e^{nxz}}{(k!)^n}
\sum_{\nu=0}^{nk}z^{nk-\nu}\sum_{j=1}^{nk}d_j^{(\nu)}e^{jx}\label{3.7}\\
&= \frac{1}{(k!)^n}\sum_{\nu=0}^{nk}z^\nu\sum_{j=0}^{nk}d_j^{(nk-\nu)}
\left[\bigg(\frac{x}{e^x-1}\bigg)^{n(k+1)}e^{(nz+j)x}\right],\nonumber
\end{align}
where $d_0^{(0)}=1$ and $d_0^{(\nu)}=0$ for $1\leq\nu\leq nk$.

At this point we use the Bernoulli polynomials of order $r$,
defined by the generating function
\begin{equation}\label{3.8}
\bigg(\frac{x}{e^x-1}\bigg)^re^{zx} 
= \sum_{m=0}^\infty B_m^{(r)}(z)\frac{x^m}{m!},\qquad |x|<2\pi;
\end{equation}
see, e.g., \cite[p.~127]{MT}. Usually $r$ is a positive integer, but it can 
also be considered as a variable. In particular, comparing \eqref{3.8} with 
\eqref{1.2} shows that $B_m^{(1)}(z)=B_m(z)$.

The right-most term of \eqref{3.7} now leads to
\begin{equation}\label{3.9}
[x^{n(k+1)-1}]F_k(x,z)^n
= \frac{1}{(k!)^n(n(k+1)-1)!}\sum_{\nu=0}^{nk}z^\nu\sum_{j=0}^{nk}d_j^{(nk-\nu)}
B_{n(k+1)-1}^{(n(k+1))}(nz+j).
\end{equation}
Using \eqref{2.2} and the fact that 
\begin{equation}\label{3.10}
B_{m-1}^{(m)}(y)=(y-1)(y-2)\cdots(y-m+1)
\end{equation}
(see, e.g., \cite[p.~130]{MT}), we get
\begin{equation}\label{3.11}
S_{n-1,k}(z) 
= \frac{1}{k!(n(k+1)-1)!}\sum_{\nu=0}^{nk}z^\nu\sum_{j=0}^{nk}d_j^{(nk-\nu)}
\prod_{r=1}^{n(k+1)-1}(nz+j-r).
\end{equation}
The product on the right is always divisible by $z$ when $j\geq 1$. When $j=0$,
then $d_j^{(nk-\nu)}=0$ for all $\nu < nk$, while for $\nu=nk$ we have the 
factor $z^{nk}$ in the corresponding summand on the right of \eqref{3.11}. 
Altogether, $S_{n-1,k}(z)$ is therefore divisible by $z$,
and by \eqref{1.4} is also divisible by $z-1$.

Finally, we consider again the product on the right-hand side of \eqref{3.11},
for $j=0,1,\ldots,nk$, and display it as follows:
\begin{align*}
j=0:\quad &(nz-1)(nz-2)\cdots(nz-n(k+1)+1),\\
j=1:\quad &nz(nz-1)\cdots(nz-n(k+1)+2),\\
&\vdots \\
j=nk:\quad &(nz+nk-1)(nz+nk-2)\cdots(nz-n+1).
\end{align*}
We now see that each of the factors $nz-1$, $nz-2$, \ldots, $nz-n+1$ occurs in
all $nk+1$ products. This means that, by \eqref{3.11}, the polynomial 
$S_{n-1,k}(z)$ is divisible by the product of these $n-1$ terms. Replacing,
finally, $n$ by $n+1$, this completes the proof of Theorem~1.
\end{proof}

\section{Some further results}

We return to the polynomials $S_{k,\nu}^{(n)}(y)$, defined as part of the 
identity \eqref{3.2} by
\begin{equation}\label{4.a}
S_{k,\nu}^{(n)}(y) = \sum_{\substack{j_1,\ldots,j_n\geq 0\\j_1+\dots+j_n=\nu}}
\binom{k}{j_1}A_{j_1}(y)\cdots\binom{k}{j_n}A_{j_n}(y).
\end{equation}

Some of the coefficients of these polynomials are relatively easy to
determine.

\begin{lemma}\label{lem:5}
For $\nu\geq 1$ we have
\[
S_{k,\nu}^{(n)}(y) = \sum_{j=1}^\nu c_jy^j,
\]
where
\[
c_1 = n\binom{k}{\nu},\quad
c_2 = \binom{n}{2}\left[\binom{2k}{\nu}-2\binom{k}{\nu}\right]
+n\left(2^\nu-\nu-1\right)\binom{k}{\nu},\quad
c_\nu = \binom{nk}{\nu}.
\]
\end{lemma}

\begin{proof}
By the known properties of the Eulerian polynomials, the lowest power of $y$ in
$A_k(y)$ for all $k\geq 1$ is $y^1$, and the coefficient is always 1. Hence 
we get a 
contribution to the coefficient of $y$ in $S_{k,\nu}^{(n)}(y)$ if and only if
all but one of the summation indices $j_i$ are 0, and one has to be $\nu$.
Since there are $n$ such cases, the coefficient of $y$ is $n\binom{k}{\nu}$,
as claimed.

Next, to determine the coefficient $c_2$, we need to consider two possibilities
for the defining sum \eqref{4.a}. First, we assume that all but two of the 
$j_i$ are 0, say $j_3=\cdots=j_n=0$, while $j_1\geq 1$ and $j_2\geq 1$. Then
$j_1+j_2=\nu$, and thus $1\leq j_1,j_2\leq \nu-1$. Also, the coefficients of
$y$ in $A_{j_i}(y)$ are 1. Hence with this assumption the sum becomes
\[
\sum_{r=1}^{\nu-1}\binom{k}{r}\binom{k}{\nu-r}=\binom{2k}{\nu}-2\binom{k}{\nu},
\]
where we have used the Chu-Vandermonde convolution; see, e.g., \cite[p.~8]{Ri}.
Since there are $\binom{n}{2}$ ways of selecting two nonzero $j_i$, the total
contribution is
\begin{equation}\label{4.b}
\binom{n}{2}\left[\binom{2k}{\nu}-2\binom{k}{\nu}\right].
\end{equation}

Second, we assume that all but one $j_i$ are 0. In this case the contribution
is $n\binom{k}{\nu}$ times the coefficient of $y^2$ in $A_\nu(y)$, namely
$A(\nu,2)=2^\nu-\nu-1$, where this evaluation of the Eulerian number can be
found, e.g., in \cite[p.~243]{Co}. This, combined with \eqref{4.b}, gives
the coefficient $c_2$.

Finally, since all $A_j(y)$ have degree $j$ and leading coefficient
1, and since $j_1+\cdots+j_n=\nu$, the polynomial $S_{k,\nu}^{(n)}(y)$ has
degree $\nu$ and leading coefficient
\[
\sum_{\substack{j_1,\ldots,j_n\geq 0\\j_1+\dots+j_n=\nu}}
\binom{k}{j_1}\cdots\binom{k}{j_n} = \binom{nk}{\nu},
\]
where we have used the generalized Vandermonde identity. This 
completes the proof of the lemma.
\end{proof}

We note that Lemma~\ref{lem:5} is consistent with the four examples given after
\eqref{3.2}. Another easy property of the polynomials $S_{k,\nu}^{(n)}(y)$
will be required in the proof of the next theorem.

\begin{lemma}\label{lem:7}
For positive integers $n$ and $k$ we have
\begin{equation}\label{4.1}
S_{k,nk}^{(n)}(y) = A_k(y)^n\qquad\hbox{and}\qquad y^n\mid S_{k,nk}^{(n)}(y),
\end{equation}
where $A_k(y)$ is the $k$th Eulerian polynomial defined by \eqref{2.6}.
\end{lemma}

\begin{proof}
For $\nu=nk$, the only summand in the defining multiple sum for 
$S_{k,\nu}^{(n)}(y)$ that does not vanish corresponds to $j_1=\cdots j_n=k$, 
which implies the first part of \eqref{4.1}. The second part follows from the
fact that $y\mid A_k(y)$ for all $k\geq 1$.
\end{proof}

To motivate the following result, we note that Table~1 seems to indicate that 
the polynomial $S_{n,k}(z)$, defined in \eqref{1.3}, is divisible by 
$z^2(z-1)^2$ when $k=2$ and $n$ is even. This is in fact true in general.

\begin{theorem}\label{thm:6}
If $n, k$ are positive even integers, then $S_{n,k}(z)$ is divisible by
$z^2(z-1)^2$.
\end{theorem}

\begin{proof}
By \eqref{3.11} we only need to consider the summand for $\nu=0$, and we are
done if we can show that
\begin{equation}\label{4.2}
T_{n,k}(z) := \sum_{j=0}^{nk}d_j^{(nk)}\prod_{r=1}^{n(k+1)-1}(nz+j-r)
\end{equation}
is divisible by $z^2$ when $k$ is even and $n$ is odd (note the shift in $n$ in
\eqref{3.11}). By \eqref{3.6} with $\nu=nk$ and by \eqref{4.1} we have
$d_j^{(nk)}=0$ for $0\leq j\leq n-1$, while $a_j:=d_{n+j}^{(nk)}$, 
$j=0,1,\ldots,nk-n$ are the coefficients of $(A_k(y)/y)^n$, which is a 
polynomial in $y$ of degree $n(k-1)$. Since $A_k(y)/y$ is a self-reciprocal 
polynomial, then so is $(A_k(y)/y)^n$, and therefore 
$a_0, a_1, \ldots, a_{nk-n}$ is a palindromic
sequence, i.e., $a_j=a_{m-j}$, where for convenience we have set $m:=nk-n$. 
We can then rewrite \eqref{4.2} as
\begin{equation}\label{4.2a}
T_{n,k}(z) = \sum_{j=n}^{nk}a_{j-n}\prod_{r=1}^{n(k+1)-1}(nz+j-r) 
= \sum_{j=0}^m a_j\prod_{r=1}^{n(k+1)-1}(nz+n+j-r),
\end{equation}
which holds for any integers $n,k\geq 1$. Assuming now that $k$ is even and $n$
is odd, and using symmetry, we get 
\[
T_{n,k}(z) = \sum_{j=0}^{\lfloor m/2\rfloor} a_j
\bigg(\prod_{r=1}^{n(k+1)-1}(nz+n+j-r)+\prod_{r=1}^{n(k+1)-1}(nz+nk-j-r)\bigg),
\]
where in the second product $j$ was replaced by $nk-n-j$. We also note that
$m=n(k-1)$ is odd since $n$ is odd and $k$ is even; hence the original $m+1$
summands divide evenly into the pairs in the last equation above.

Now we switch the order of multiplication in the second product, replacing
$r$ by $nk+n-r$. Then we get
\begin{equation}\label{4.3}
T_{n,k}(z) = \sum_{j=0}^{\lfloor m/2\rfloor} a_j
\bigg(\prod_{r=1}^{n(k+1)-1}(nz+n+j-r)+\prod_{r=1}^{n(k+1)-1}(nz-(n+j-r))\bigg).
\end{equation}
For each $j$, $0\leq j\leq\lfloor m/2\rfloor$, there is an $r$ with 
$1\leq r\leq n(k+1)-1$, such that $r=n+j$. Hence the two products in \eqref{4.3}
are divisible by $nz$. Next we note that the coefficients of $nz$ in the two
products are 
\begin{equation}\label{4.4}
\prod_{\substack{r=1\\r\neq n+j}}^{n(k+1)-1}(n+j-r)\quad\hbox{and}\quad
\prod_{\substack{r=1\\r\neq n+j}}^{n(k+1)-1}(-(n+j-r)),
\end{equation}
respectively. Since $n(k+1)$ is odd, then so is the number of factors, 
$n(k+1)-2$, in both products \eqref{4.4}, which are therefore negatives of 
each other. Hence the coefficient of $nz$ in \eqref{4.3} is 0. Replacing
$n$ by $n+1$, this proves that $S_{n,k}(z)$ is divisible by $z^2$. 
Finally, divisibility by $(z-1)^2$ now follows from \eqref{1.4}.
\end{proof}

By refining the method of proof of Theorem~\ref{thm:6} one can, in principle,
determine the coefficient of $z$ in $S_{n,k}(z)$, as defined in \eqref{1.3}, 
in the case where one of $n$ and $k$ is odd. 

\begin{theorem}\label{thm:8}
Let $n$ and $k$ be positive integers, not both even. Then the coefficient of
$z$ in $S_{n,k}(z)$ is
\begin{equation}\label{4.5}
\frac{(-1)^{k(n+1)-1}(n+1)}{k!((k+1)(n+1)-1)!}
\sum_{j=0}^{(k-1)(n+1)}(-1)^ja_j^{(k,n+1)}(n+j)!(k(n+1)-1-j)!,
\end{equation}
where $a_j^{(k,n+1)}$ is the coefficient of $y^j$ in $(A_k(y)/y)^{n+1}$, with
$A_k(y)$ the $k$th Eulerian polynomial, defined in \eqref{2.6}.
\end{theorem}

\begin{proof}
By \eqref{3.11} and \eqref{4.2a}, the coefficient of $z$ in $S_{n-1,k}(z)$ is
the coefficient of $z$ in the expression
\begin{equation}\label{4.7}
\frac{1}{k!((k+1)n-1)!}
\sum_{j=0}^{(k-1)n}a_j^{(k,n)}\prod_{r=1}^{(k+1)n-1}(nz+n+j-r).
\end{equation}
Now, the coefficient of $z$ in the product in this expression is
\begin{align*}
n\prod_{r=1}^{n+j-1}(n+j-r)&\cdot\prod_{r=n+j+1}^{(k+1)n-1}(n+j-r)\\
&= n\prod_{r=1}^{n+j-1}r\cdot\prod_{r=1}^{kn-j-1}(-r)\\
&= (-1)^{kn-j-1}n(n+j-1)!(kn-j-1)!.
\end{align*}
Finally we combine this with \eqref{4.7}, and replace $n$ by $n+1$. This
immediately gives the expression \eqref{4.5}.
\end{proof}

For the cases $k=1$ and $k=2$ in Theorem~\ref{thm:8} we can actually obtain
explicit expressions.

\begin{corollary}\label{cor:9}
$(a)$ For any $n\geq 1$, the coefficient of $z$ in $S_{n,1}(z)$ is
\[
\frac{(-1)^n}{\binom{2n+1}{n}}.
\]
$(b)$ For any odd $n\geq 1$, the coefficient of $z$ in $S_{n,2}(z)$ is
\[
-\frac{(n+1)!^2(n+\frac{n+1}{2})!}{2(3n+2)!(\frac{n+1}{2})!}.
\]
\end{corollary}

\begin{proof}
(a) For $k=1$, the sum in \eqref{4.5} consists of a single summand, and since
$(A_1(y)/y)^{n+1}=1$ for all $n$, the expression \eqref{4.5} becomes
\[
\frac{(-1)^n(n+1)}{(2n+1)!}\cdot n!^2 = (-1)^n/\binom{2n+1}{n},
\]
as claimed.

(b) Since $A_2(y)/y=1+y$, we have
\[
a_j^{(2,n+1)} = \binom{n+1}{j},\qquad j=0, 1, \ldots, n+1,
\]
and the expression \eqref{4.5} with $k=2$ becomes
\begin{align}
&\frac{-(n+1)(n+1)!}{2(3n+2)!}
\sum_{j=0}^{n+1}(-1)^j\frac{(n+j)!(2n+1-j)!}{j!(n+1-j)!}\label{4.6} \\
&\qquad =\frac{-(n+1)(n+1)!n!^2}{2(3n+2)!}
\sum_{j=0}^{n+1}(-1)^j\binom{n+j}{j}\binom{2n+1-j}{n+1-j}.\nonumber
\end{align}
By identity (3.36) in \cite{Go}, this last binomial sum is 0 when $n$ is even
(consistent with Theorem~\ref{thm:6}), and is $\binom{n+(n+1)/2}{(n+1)/2}$
when $n$ is odd. Combining this with \eqref{4.6} gives the result of part (b).
\end{proof}

The coefficient of $z$ in $S_{n,1}(z)$ is obviously the reciprocal of an integer
for each $n\geq 1$. Given the form of the corresponding expression in 
Corollary~\ref{cor:9}(b), it is rather surprising that this should also hold
for the coefficients of $z$ in $S_{n,2}(z)$. In fact, we prove slightly more.

\begin{corollary}\label{cor:10}
For any integer $n\geq 0$, denote
\begin{equation}\label{4.11}
c_n:=\frac{(n+1)!^2(n+\frac{n+1}{2})!}{2(3n+2)!(\frac{n+1}{2})!},
\end{equation}
which is defined for all integers $n\geq 0$ if we interpret fractional 
factorials in terms of the gamma function. Then $c_n$ is the reciprocal of
an even integer for all $n\geq 0$.
\end{corollary}

Before proving this result, we give the first few terms, namely
\[
c_0=\frac{1}{4},\quad c_1=\frac{1}{30},\quad c_2=\frac{1}{256},\quad
c_3=\frac{1}{2310},\quad c_4=\frac{1}{21504}.
\]
The sequence of reciprocals 4, 30, 256, 2310, \ldots, is listed as sequence
A091527 in \cite{OEIS}, with the explicit expansion (in our notation)
\begin{equation}\label{4.12}
\frac{1}{c_{n-1}} 
= \frac{(3n)!\Gamma(\tfrac{n}{2}+1)}{n!^2\Gamma(\tfrac{3n}{2}+1)}.
\end{equation}

\begin{proof}[Proof of Corollary~\ref{cor:10}]
It is easy to verify that \eqref{4.12} is consistent with \eqref{4.11}, and
using Euler's duplication formula, $c_n$ can be simplified as
\[
c_n = 2^{-2n-2}\frac{\Gamma(\tfrac{n}{2}+1)\Gamma(n+2)}{\Gamma(\tfrac{3n}{2}+2)}
= 2^{-2n-2}\binom{\tfrac{3n}{2}+1}{n+1}^{-1}.
\]
For even $n$, the last term contains a binomial coefficient, which gives the
statement of the corollary in this case. For odd $n$, on the other hand, we are
dealing with a generalized binomial coefficient which in general is not an 
integer. We wish to show that for odd $n$ the expression
\[
d_n := 2^{2n+2}\binom{\tfrac{3n}{2}+1}{n+1}
\]
is an integer. We use the fact that binomial coefficients, including generalized
ones, satisfy the recurrence relation
\[
\binom{\tfrac{3n}{2}+1}{n+1}
= \binom{\tfrac{3n}{2}}{n+1} + \binom{\tfrac{3n}{2}}{n},
\]
so the binomial coefficient on the left can be reduced to a linear combination
of terms of the form $\binom{1/2}{k}$, with $k\leq n+1$. Hence we are done if
we can show that 
\[
2^{2n+2}\binom{1/2}{k}\in{\mathbb Z},\qquad k=0,1,\ldots, n+1.
\]
By definition, $\binom{1/2}{0}=1$, and for $k\geq 1$ we have
\begin{align*}
\binom{1/2}{k} 
&= \frac{(\tfrac{1}{2})(\tfrac{1}{2}-1)\cdots(\tfrac{1}{2}-k+1)}{k!}
=\frac{(-1)^{k-1}}{2^kk!}\cdot 1\cdot 3\cdots (2k-3)\\
&= \frac{(-1)^{k-1}}{2^kk!}\cdot\frac{(2k-2)!}{2^{k-1}k!}
= \frac{(-1)^{k-1}}{2^{2k-1}}\cdot\frac{1}{k}\binom{2k-2}{k-1}.
\end{align*}
Now $C_{k-1}=\frac{1}{k}\binom{2k-2}{k-1}$ is a Catalan number, which is an
integer. With this, we finally have
\[
2^{2n+2}\binom{1/2}{k} = (-1)^{k-1}2^{2(n+1-k)+1}C_{k-1}.
\]
This shows that $d_n$ is an even integer when $n$ is odd, and thus for all $n$.
\end{proof}

\noindent
{\bf Remarks.} (1) Corollary~\ref{cor:10} can actually be improved as follows.
We set
\[
a_n:=\frac{1}{2(3n+2)c_n}
=\frac{(3n+1)!(\tfrac{n+1}{2})!}{(n+1)!^2(n+\tfrac{n+1}{2})!}
=\frac{2^{2n}}{n+1}\binom{3n/2}{n},
\]
where the right-most equality is obtained as in the proof of 
Corollary~\ref{cor:10}. The generating function
\begin{equation}\label{4.13}
A(z) := \sum_{n=0}^\infty a_nz^n
\end{equation}
can then be computed as
\[
A(z)=\frac{1}{2z}\left(1-\cos(\tfrac{1}{3}\arcsin(6\sqrt{3}z))
+\tfrac{1}{\sqrt{3}}\sin(\tfrac{1}{3}\arcsin(6\sqrt{3}z))\right),
\]
and with some effort one can show that $A(z)$ satisfies the equation
\[
2z^2A(z)^3 - 3zA(z)^2 + A(z) = 1.
\]
Substituting $A(z)$ from \eqref{4.13} into this equation, expanding, and then
equating coefficients of $z^n$, we see that each $a_n$ can be written as a sum
of products of $a_k$ with $0\leq k\leq n-1$, with integer coefficients. This
implies that all the $a_n$ are integers, and thus the denominator of $c_n$ 
is not only even, but is in fact
divisible by $2(3n+2)$. The first few terms of the sequence $(a_n)$ are
\[
a_0=1,\quad a_1=3,\quad a_2=16,\quad a_3=105,\quad a_4=768.
\]
This sequence can be found as A085614 in \cite{OEIS}; according to this entry,
$a_n$ is the number of elementary arches of size $n+1$.

\medskip
(2) Since the coefficients of $z$ in $S_{n,1}(z)$ and $S_{n,2}(z)$ are 
reciprocals of integers, it would be interesting to know whether this is also
the case for $S_{n,3}(z)$. By direct computation using Theorem~\ref{thm:8}, we
find that the first few coefficients of $z$ in $S_{n,3}(z)$ are
\[
c_1^{(3)}=-\frac{1}{126},\quad c_2^{(3)}=-\frac{1}{1155},\quad
c_3^{(3)}=-\frac{1}{6930},\quad\hbox{and}\quad c_4^{(3)}=-\frac{10}{513513},
\]
which answers the question in the negative. Computing the first 30 terms and 
using a recurrence fitting algorithm, we found the recurrence relation
\[
12(4n+5)(4n+11)c_{n+2}^{(3)}-8(n+3)(2n+3)c_{n+1}^{(3)}-(n+2)(n+3)c_n^{(3)}=0.
\]
To prove this, we use Theorem~\ref{thm:8} with $k=3$, and write $a_j^{(3,n+1)}$ 
as a binomial sum. Then we get
\[
c_n^{(3)}= \frac{n+1}{6(4n+3)!}\sum_{j=0}^{2n+2}(-1)^{n+j}(n+j)!(3n+2-j)!
\sum_{i=0}^{n+1}\binom{n+1}{i}\binom{n+1-i}{j-2i}4^{j-2i}.
\]
Finally we apply ``creative telescoping" (twice) to this combinatorial sum,
for instance by using C.~Koutschan's Mathematica package 
\texttt{HolonomicFunctions} \cite{Ko}.

\section{Additional remarks and questions}

{\bf 1.} Using the identity \eqref{3.10}, which involves a special higher-order
Bernoulli polynomial, and using the notation of \eqref{1.3}, we can rewrite the
identity \eqref{1.1} as
\[
S_{n,0}(z) = \frac{1}{n!}B_n^{(n+1)}((n+1)z).
\]
Similarly, the second part of Theorem~\ref{thm:1} can be rephrased to state 
that for all positive integers $n$ and $k$, the polynomial
$z(z-1)B_n^{(n+1)}((n+1)z)$ divides $S_{n,k}(z)$.

\medskip
{\bf 2.} Related to this, if we divide $S_{n,1}(z)$ by 
$z(z-1)B_n^{(n+1)}((n+1)z)$, we get a sequence of polynomials of degree at
most $n-1$; see Table~1. The first few of these polynomials are listed in
Table~3, normalized so that their constant coefficients are 1; we denote 
them by $p_n(x)$. 

\bigskip
\begin{center}
{\renewcommand{\arraystretch}{1.2}
\begin{tabular}{|r|l|}
\hline
$n$ & $p_n(z)$ \\
\hline
1 & $1$ \\
2 & $1-2z$ \\
3 & $1-z+z^2$ \\
4 & $1+\frac{1}{6}z-\frac{13}{2}z^2+\frac{13}{3}z^3$ \\
5 & $1+\frac{3}{2}z-\frac{27}{2}z^2+24z^3-12z^4$ \\
6 & $1+\frac{179}{60}z-\frac{473}{24}z^2+29z^3-\frac{571}{24}z^4+\frac{571}{60}z^5$ \\
\hline
\end{tabular}}

\medskip
{\bf Table~3}: $p_n(z)$ for $1\leq n\leq 6$.
\end{center}

\bigskip
It follows from Theorem~\ref{thm:1} that these polynomials satisfy the symmetry
property $p_n(1-z)=(-1)^{n-1}p_n(z)$, which in turn implies that $p_n(z)$ is
divisible by $2z-1$ when $n$ is even. Apart from this, can anything else be 
said about the polynomials $p_n(z)$? The fact that a relatively large prime 
(namely 571) appears in the leading coefficient of
$p_6(z)$ seems to indicate that the leading coefficients of these polynomials,
and indeed of the polynomials $S_{n,k}(z)$, are not as straightforward as the
coefficients of $z$ (see, especially, Corollary~\ref{cor:9}).

\medskip
{\bf 3.} Again related to the previous point, we note that each product of 
Bernoulli polynomials in \eqref{1.3} has degree 
\[
\sum_{i=1}^m(k+1+j_i) = m(k+1)+(k+1)(n-m)+k = (k+1)(n+1)-1,
\]
which is independent of $m$. When $k$ is even, then all terms in \eqref{1.3} are 
positive, and since the Bernoulli polynomials are monic, we may conclude that
the degree of $S_{n,k}(z)$ is also $(k+1)(n+1)-1$. However, this is not clear
when $k$ is odd and $S_{n,k}(z)$ is therefore an alternating sum in $m$.

\medskip
{\bf 4.} We will now see that there is a close connection between the function
$F_k(x,z)$, defined by \eqref{2.1}, and the {\it Lerch transcendent} (also 
known as the {\it Lerch zeta function\/})
\begin{equation}\label{5.a}
\Phi(z,s,v) := \sum_{n=0}^\infty\frac{z^n}{(v+n)^s},
\end{equation}
with $|z|<1$ and $v\neq 0, -1, -2,\ldots$. For various properties and 
identities see, e.g., \cite[Sect.~1.1]{EMOT}. One of these identities is
\begin{equation}\label{5.b}
\Phi(z,-m,v) = \frac{m!}{z^v}\big(\log\frac{1}{z}\big)^{-m-1} - \frac{1}{z^v}
\sum_{r=0}^\infty\frac{B_{m+r+1}(v)}{m+r+1}\cdot\frac{(\log{z})^r}{r!}, 
\end{equation}
where $m$ is a positive integer and $|\log{z}|<2\pi$ (see \cite[p.~30]{EMOT}.)
Replacing $m, r, v$ and $z$ by $k, m ,z$ and $e^x$, respectively, \eqref{5.b}
immediately yields
\[
(-x)^{k+1}\frac{e^{xz}}{k!}\Phi(e^x,-k,z) = 1+(-1)^k\frac{x^{k+1}}{k!}
\sum_{m=0}^\infty\frac{B_{m+k+1}(z)}{m+k+1}\frac{x^m}{m!},
\]
and thus, with \eqref{2.1} we have
\begin{equation}\label{5.c}
F_k(x,z) = (-x)^{k+1}\frac{e^{xz}}{k!}\Phi(e^x,-k,z).
\end{equation}
Without going into further details, we mention that Boyadzhiev 
\cite[Eq.~(6.5)]{Bo} expresses Lerch transcendents such as the ones on the 
right of \eqref{5.c} in terms of Apostol-Bernoulli polynomials, which in turn
can be written as sums involving Eulerian polynomials; see \cite[Eq.~(4.7)]{Bo}.
Putting everything together, we get \eqref{2.9} again, as expected.
Furthermore, combining \eqref{5.c} with \eqref{5.a}, we get
\begin{equation}\label{5.d}
F_k(x,z) = \frac{(-x)^{k+1}}{k!}\sum_{n=0}^\infty(n+z)^ke^{(n+z)x},
\end{equation}
while \eqref{2.11} gives
\[
F_k(-x,1-z) = \frac{(-x)^{k+1}}{k!}\sum_{m=1}^\infty(m-1+z)^ke^{(1-z-m)x}.
\]
The right-hand side of this last identity is the same as that of \eqref{5.d}, 
and by \eqref{2.1a} the two identities are consistent.

\medskip
{\bf 5.} In addition to Corollary~\ref{cor:9}, one can say a little more about
the coefficients $a_j^{(k,n)}$ that occur in Theorem~\ref{thm:8}. Indeed, to
find an expression for $a_j^{(k,n)}=[y^j](A_k(y)/y)^n$, we apply the multinomial
theorem to \eqref{2.7}, obtaining
\begin{align*}
\frac{A_k(y)^n}{y^n} &= \sum_{j_1+j_2+\cdots+j_k=n}\binom{n}{j_1,\ldots,j_k}
A(k,1)^{j_1}\big(A(k,2)y\big)^{j_2}\cdots \big(A(k,k)y^{k-1}\big)^{j_k}\\
&= \sum_{j_1+j_2+\cdots+j_k=n}\binom{n}{j_1,\ldots,j_k}
A(k,1)^{j_1}\cdots A(k,k)^{j_k}y^{j_2+2j_3+\cdots+(k-1)j_k},
\end{align*}
and therefore
\begin{equation}\label{5.1}
\bigg(\frac{A_k(y)}{y}\bigg)^n 
= \sum_{j=0}^{n(k-1)}\left(\sum\binom{n}{j_1,\ldots,j_k}A(k,1)^{j_1}
\cdots A(k,k)^{j_k}\right)y^j,
\end{equation}
where the inner sum is taken over all $j_1,\ldots, j_k$ satisfying
\begin{equation}\label{5.2}
\begin{cases}
j_1+j_2+\cdots+j_k=n,\\
j_2+2j_3+\cdots+(k-1)j_k=j,
\end{cases}
\end{equation}
Now, by definition and \eqref{5.2} we have
\[
a_j^{(k,n)} 
= \sum\binom{n}{j_1,\ldots,j_k}A(k,1)^{j_1}A(k,2)^{j_2}\cdots A(k,k)^{j_k},
\]
where the summation is again over all $j_1,\ldots, j_k$ satisfying \eqref{5.2}.
This identity can be rewritten as
\begin{equation}\label{5.3}
a_j^{(k,n)} = \sum_{j_1=0}^n\frac{n!}{j_1!(n-j_1)!}
\sum\binom{n-j_1}{j_2,\ldots,j_k}A(k,2)^{j_2}\cdots A(k,k)^{j_k},
\end{equation}
where we have used the fact that $A(k,1)=1$, and the inner sum is over all 
$j_2,\ldots, j_k$ satisfying
\[
\begin{cases}
j_2+\cdots+j_k=n-j_1,\\
j_2+2j_3+\cdots+(k-1)j_k=j.
\end{cases}
\]
Finally, setting $r:=n-j_1$, we can slightly simplify \eqref{5.3} as
\begin{equation}\label{5.4}
a_j^{(k,n)} = \sum_{r=0}^n\binom{n}{r}
\sum\binom{r}{j_2,\ldots,j_k}A(k,2)^{j_2}\cdots A(k,k)^{j_k},
\end{equation}
where the inner sum is over all $j_2,\ldots, j_k$ satisfying
\[
\begin{cases}
j_2+\cdots+j_k=r,\\
j_2+2j_3+\cdots+(k-1)j_k=j.
\end{cases}
\]
This last inner sum is reminiscent of an (ordinary) Bell polynomial, but
is still somewhat different. 

\medskip
{\bf 6.} For a different approach to the coefficients $a_j^{(k,n)}$ we use the
infinite series \eqref{2.8} and rewrite it as
\begin{equation}\label{5.5}
\frac{\frac{1}{y}A_k(y)}{(1-y)^{k+1}} = \sum_{j=0}^\infty (j+1)^ky^j,
\end{equation}
for $k\geq 1$. Raising both sides of \eqref{5.5} to the $n$th power, we get
\begin{equation}\label{5.6}
\frac{(A_k(y)/y)^n}{(1-y)^{(k+1)n}} = \sum_{\nu=0}^\infty\bigg(
\sum_{\substack{j_1,\ldots,j_n\geq 0\\j_1+\cdots+j_n=\nu}}\big((j_1+1)\cdots
(j_n+1)\big)^k\bigg)y^\nu. 
\end{equation}
We denote the inner multiple sum by $u_\nu^{(k,n)}$ and rewrite it as
\begin{equation}\label{5.7}
u_\nu^{(k,n)} = \sum_{\substack{j_1,\ldots,j_n\geq 1\\j_1+\cdots+j_n=\nu+n}}
\big(j_1\cdots j_n\big)^k.
\end{equation}
Using \eqref{5.6}, \eqref{5.7} and the definition of $a_j^{(k,n)}$, we get
\[
a_j^{(k,n)} = \sum_{\nu=0}^j(-1)^{j-\nu}\binom{(k+1)n}{j-\nu}u_\nu^{(k,n)}.
\]
It is therefore of interest to find out more about the numbers $u_\nu^{(k,n)}$.

\end{document}